\documentclass[11pt,a4paper]{amsart}
\usepackage{amssymb}
\usepackage{amsfonts}% lo he annadido
\usepackage{graphicx}
\usepackage{epsfig}
\usepackage{mathtools}
\usepackage{amsmath}
\makeatletter
\newcommand\Label[1]{&\refstepcounter{equation}(\theequation)\ltx@label{#1}&}
\makeatother

\theoremstyle{plain}
\newtheorem{theorem}{Theorem}[section]
\newtheorem{proposition}[theorem]{Proposition}
\newtheorem{lemma}[theorem]{Lemma}
\newtheorem{corollary}[theorem]{Corollary}
\theoremstyle{definition}

\newtheorem{remark}[theorem]{Remark}

\newtheorem{case}{Case}

\newcommand{\R}{\mathbb{R}}
\newcommand{\C}{\mathbb{C}}

\newcommand{\bR}{{\mathbb R}}

\def\H{\mathbb{H}}

\def\<{\langle}
\def\>{\rangle}
\def\ch#1{{{\bf H}^{#1}_{\C}}}
\def\hh#1{{{\bf H}^{#1}_{\H}}}
\def\P{\mathbb P}
\begin{document}
\title[Quaternionic Kleinian groups]{A characterization of quaternionic Kleinian groups in dimension $2$ with complex trace fields}
\author{Sungwoon Kim}
\author{Joonhyung Kim}

   \address{Department of Mathematics, Jeju National University, Jeju, 690-756, Republic of Korea}
   \email{sungwoon@kias.re.kr}

\address{Joonhyung Kim \\
    Department of Mathematics Education, Hannam University\\
           70 Hannam-ro, Daedeok-gu\\
           Daejeon 306-791, Republic of Korea}
\email{calvary\char`\@snu.ac.kr}

        \date{}
        \maketitle

\begin{abstract}
Let $G$ be a non-elementary discrete subgroup of $\mathrm{Sp}(2,1)$. We show that if the sum of diagonal entries of each element of $G$ is a complex number, then $G$ is conjugate to a subgroup of $\mathrm{U}(2,1)$.
%In this paper, we characterize a non-elementary discrete subgroup of $\mathrm{Sp}(2,1)$. Precisely, for a non-elementary quaternionic hyperbolic Kleinian group $G < \mathrm{Sp}(2,1)$ containing a loxodromic element fixing $0$ and $\infty$, if the sum of diagonal entries of each element of $G$ is in $\C$, then we show that $G$ is conjugate to a subgroup of $\mathrm{SU}(2,1)$. 
\end{abstract}
\footnotetext[3]{This work was supported by the Project ``PoINT" of Jeju National University in 2016.}

\section{Introduction}

Given a Kleinian group $G$ of $\mathrm{PSL}(2,\mathbb C)$, its trace field, denoted by $\mathbb Q(\text{tr} G)$, is defined as the field generated by the traces of its elements.
The trace fields have played an important role in studying arithmetic aspects of Kleinian groups.
Neumann and Reid \cite{NR} have studied the trace fields of arithmetic lattices in $\mathrm{PSL}(2,\mathbb C)$.
They showed that if $G$ is a non-uniform arithmetic lattice, it is conjugate to a subgroup of $\mathrm{PSL}(2, \mathbb Q(\text{tr}G))$.

Even if the notion of trace field was first defined for Kleinian groups in $\mathrm{PSL}(2,\mathbb C)$, it is possible to extend the notion to complex and quaternionic Kleinian groups.
Indeed there have been a few studies concerning the trace fields of complex and quaternionic Kleinian groups.
Most of studies on the trace fields of complex Kleinian groups have focused on extending the results in the case of $\mathrm{PSL}(2,\mathbb C)$ to $\mathrm{SU}(n,1)$.
McReynolds \cite{Mc} showed that the trace fields of complex Kleinian groups are commensurability invariants as for real Kleinian groups.
Cunha-Gusevskii \cite{CG} and Genzmer \cite{Ge} studied whether a discrete subgroup of $\mathrm{SU}(2,1)$ can be realized over its trace field.

A central theme in studying the trace fields of complex Kleinian groups is to characterize complex Kleinian groups with real trace fields.
It turns out that any non-elementary complex Kleinian group with real trace field preserves a totally geodesic submanifold of constant negative sectional curvature in the complex hyperbolic space.
Cunha-Gusevskii \cite{CG} and  Fu-Li-Wang \cite{FLW12} proved this for Kleinian groups in $\mathrm{SU}(2,1)$, and then Kim-Kim \cite{KK14} extended it to $\mathrm{SU}(3,1)$.
Recently J. Kim and S. Kim \cite{KK} extended this result to $\mathrm{SU}(n,1)$ in general. Furthermore they showed that any non-elementary quaternionic Kleinian group with real trace field 
is also conjuate to a subgroup of either $\mathrm{SO}(n,1)$ or $\mathrm{SU}(1,1)$.

For quaternionic Kleinian groups, J. Kim \cite{JKim} proved that if a non-elementary quaternionic Kleinian group $G$ in $\mathrm{Sp}(3,1)$ has a loxodromic element fixing $0$ and $\infty$, and the sum of diagonal entries of each element of $G$ is real, then $G$ preserves a totally geodesic submanifold of constant negative sectional curvature in the quaternionic hyperbolic space. Then the result is extended to general $\mathrm{Sp}(n,1)$ case by J. Kim and S. Kim \cite{KK}.

The studies so far have focused on characterizing non-elementary discrete groups with real trace fields.
It is very natural to ask what if the ``real'' is replaced with ``complex''.
In this article, we give the answer for this question in the case of  $\mathrm{Sp}(2,1)$. The main theorem is the following.

%In \cite{JKim}, J. Kim proved the following theorem.
%\begin{theorem}
%Let $G < \mathrm{Sp}(2,1)$ be a non-elementary quaternionic hyperbolic Kleinian group containing a loxodromic element fixing $0$ and $\infty$. Assume that the sum of diagonal entries of each element of $G$ is real. Then $G$ is Fuchsian.
%\end{theorem}
%t is very natural to ask what if the ``real'' is replaced with ``complex'' in the assumption. In this article, we give the answer for this question. The main theorem is the following.
\begin{theorem}
Let $G < \mathrm{Sp}(2,1)$ be a non-elementary quaternionic Kleinian group containing a loxodromic element fixing $0$ and $\infty$. If the sum of diagonal entries of each element of $G$ is in $\C$, then $G$ preserves a totally geodesic submanifold of $\mathbf H_{\mathbb H}^2$ that is isometric to $\mathbf H_{\mathbb C}^2$. In other words, $G$ is conjugate to $\mathrm{U}(2,1)$.
\end{theorem}

\section{Quaternionic hyperbolic space}
The materials of this chapter are borrowed from \cite{JKim}. For basic notions, we refer \cite{JKim} for the reader and for more information, see \cite{KP}.

Let $\H^{2,1}$ be a quaternionic vector space of dimension $3$ with a
Hermitian form of signature $(2,1)$. An element of $\H^{2,1}$ is a
column vector $p=(p_1,p_2,p_3)^t$. Throughout the paper, we choose
the second Hermitian form on $\H^{2,1}$ given by a matrix
$$
J=\left[\begin{matrix} 0 & 0 & 1 \\ 0 & 1 & 0 \\ 1 & 0 & 0
\end{matrix}\right].
$$
Thus $$\<p,q\>=q^*Jp=\bar{q}^tJp=\bar{q}_1p_3+\bar{q}_2p_2+\bar{q}_3p_1,$$
where $p=(p_1,p_2,p_3)^t, q=(q_1,q_2,q_3)^t \in \H^{2,1}$.

One model of a quaternionic hyperbolic 2-space $\hh{2}$, which matches
this Hermitian form is the \emph{Siegel domain $\mathfrak{S}$}. It is defined by
identifying points of $\mathfrak{S}$ with their horospherical
coordinates, $p=(\zeta,v,u) \in \H \times \mathrm{Im}(\H) \times \R_+$. The boundary of $\mathfrak{S}$ is given by $(\H \times \mathrm{Im}(\H)) \cup \{\infty\}$, where $\infty$ is a distinguished point at infinity.
Define a map $\psi:\overline{\mathfrak{S}}\rightarrow \P \H^{2,1}$
by
$$
\psi:(\zeta,v,u)\mapsto \left[\begin{matrix}
-|\zeta|^2-u+v
\\ \sqrt{2}\zeta \\ 1\end{matrix}\right] \hbox{ for }
(\zeta,v,u)\in\overline{\mathfrak{S}}-\{\infty\} \hbox{   ; }
\psi:\infty\mapsto\left[\begin{matrix} 1
\\ 0 \\ 0\end{matrix}\right].
$$
Then $\psi$ maps $\mathfrak{S}$ homeomorphically to the set of
points $p$ in $\P \H^{2,1}$ with $\langle p,p\rangle<0$, and maps
$\partial \mathfrak{S}$ homeomorphically to the set of points $p$
in $\P \H^{2,1}$ with $\langle p,p\rangle=0$.
There is a metric on $\mathfrak{S}$ called the Bergman metric and the isometry group of $\hh{2}$ with respect to this metric is
\begin{align*}
\mathrm{Sp}(2,1) &=\{A \in \mathrm{GL}(3,\H) : \langle p,p' \rangle = \langle Ap,Ap' \rangle, p,p' \in \H^{2,1}\}\\
&= \{A \in \mathrm{GL}(3,\H) : J=A^*JA\},
\end{align*}
where $A:\H^{2,1} \rightarrow \H^{2,1}; x\H \mapsto (Ax)\H$ for $x\in \H^{2,1}$ and $A \in \mathrm{Sp}(2,1)$. As in \cite{Kim}, we adopt the convention that the action of $\mathrm{Sp}(2,1)$ on $\hh{2}$ is left and the action of projectivization of $\mathrm{Sp}(2,1)$ is right action. If we write
$$
A=\left[\begin{matrix} a & b & c \\ d & e & f \\ g & h & l
\end{matrix}\right] \in \mathrm{PSp}(2,1),
$$
$A^{-1}$ is written as
$$
A^{-1}=\left[\begin{matrix} \bar{l} & \bar{f} & \bar{c} \\ \bar{h} & \bar{e} & \bar{b} \\ \bar{g} & \bar{d} & \bar{a}
\end{matrix}\right] \in \mathrm{PSp}(2,1).
$$
Then, from $AA^{-1}=A^{-1}A=I$, we get the following identities.
\begin{align*}
& a\bar{l}+b\bar{h}+c\bar{g} =1 \Label{1},& & a\bar{f}+b\bar{e}+c\bar{d} =0 \Label{2}, & & a\bar{c}+|b|^2+c\bar{a} =0 \Label{3},  \\
& d\bar{l}+e\bar{h}+f\bar{g} =0 \Label{4},& & d\bar{f}+|e|^2+f\bar{d} =1 \Label{5}, & & d\bar{c}+e\bar{b}+f\bar{a} =0 \Label{6}, \\
& g\bar{l}+|h|^2+l\bar{g}=0 \Label{7},& & g\bar{f}+h\bar{e}+l\bar{d}=0 \Label{8},& & g\bar{c}+h\bar{b}+l\bar{a}=1 \Label{9}, \\
& \bar{l}a+\bar{f}d+\bar{c}g=1 \Label{10},& & \bar{l}b+\bar{f}e+\bar{c}h=0 \Label{11},& & \bar{l}c+|f|^2+\bar{c}l=0 \Label{12}, \\
& \bar{h}a+\bar{e}d+\bar{b}g=0 \Label{13},& & \bar{h}b+|e|^2+\bar{b}h=1 \Label{14},& & \bar{h}c+\bar{e}f+\bar{b}l=0 \Label{15}, \\
& \bar{g}a+|d|^2+\bar{a}g=0 \Label{16},& & \bar{g}b+\bar{d}e+\bar{a}h=0 \Label{17},& & \bar{g}c+\bar{d}f+\bar{a}l=1 \Label{18}.
\end{align*}

\begin{remark}\label{remark}
If $c=0$, then $f=0$ by (12) and hence $A$ fixes $\mathbf 0=[0,0,1]^t$. Similarly, if $g=0$, then $d=0$ by (16) and hence $A$ fixes $\infty=[1,0,0]^t$.
\end{remark}
%Then we say a Kleinian group $G$ in the group of holomorphic isometries of quaternionic hyperbolic space is Fuchsian if $G$ leaves invariant a complex line or $\H_{\R}^2$.\\
Note that totally geodesic submanifolds of quaternionic hyperbolic 2-space are isometric to one of $\mathbf H_{\mathbb H}^1$, $\mathbf H_{\mathbb C}^1$,  $\mathbf H_{\mathbb C}^2$, and $\mathbf H_{\mathbb R}^2$. The following proposition is essential in the proof of the main theorem.
%\begin{proposition}
%For two nonzero quaternions $a$ and $b$, if $ab \in \bR$ and $b$ is purely imaginary, then $a=\lambda b$ for some real number $\lambda$.
%\end{proposition}

%\begin{proof}
%We know that $b\neq0$ and so $b^{-1}=\bar{b}/|b|^2$. Since $b$ is purely imaginary, we have $\bar{b}=-b$. Therefore $b^{-1}=-b/|b|^2$ and $a=(ab)b^{-1}=\frac{-(ab)}{|b|^2}b$. As $-(ab)/|b|^2$ is real, we get the result.
%\end{proof}

\begin{proposition}\label{prop:commuting}
For two nonzero quaternions $a$ and $b$, if $ab$ and $ba$ are complex numbers, then $a$ and $b$ satisfy one of the following;
\begin{itemize}
\item[(i)] $a$, $b \in \C$
\item[(ii)] $a$ and $b$ are of the form $a=a_*j$ and $b=b_*j$ for some $a_*, b_* \in \C$
\item[(iii)] $b=r\bar{a}$ for some $r \in \R-\{0\}$.
\end{itemize}
\end{proposition}

\begin{proof}
Let $a=a_0+a_1i+a_2j+a_3k$ and $b=b_0+b_1i+b_2j+b_3k$. The $j$-part of $ab$ and $ba$ are $a_0b_2+a_2b_0+a_3b_1-a_1b_3$ and $a_0b_2+a_2b_0-a_3b_1+a_1b_3$ respectively. 
Since they should be zero, we have that $a_0b_2+a_2b_0=0$ and $a_3b_1-a_1b_3=0$. In a similar way, by considering $k$-parts instead of $j$-parts, it is deduced that $a_0b_3+a_3b_0=0$ and $a_1b_2-a_2b_1=0$.

When $a_0 \neq 0$, if $a_3 \neq 0$, from above identities, we get 
$$b_3=-\frac{b_0}{a_0}a_3, \ b_2=-\frac{b_0}{a_0}a_2, \  b_1=\frac{b_3}{a_3}a_1=-\frac{1}{a_3}\frac{b_0}{a_0}a_3a_1=-\frac{b_0}{a_0}a_1.$$ Hence $b=\frac{b_0}{a_0}\bar{a}$, that is, (iii) follows.
When $a_0 \neq 0$, if $a_3=0$ and moreover $a_2=0$, then $b_2$ and $b_3$ should be zero by the above identities. This means that $a$ and $b$ are in $\C$. If $a_0\neq 0$, $a_3=0$ and $a_2 \neq 0$, by the above identities, we have $$b_3=0, \ b_2=-\frac{b_0}{a_0}a_2, \ b_1=\frac{b_2}{a_2}a_1=-\frac{1}{a_2}\frac{b_0}{a_0}a_2a_1=-\frac{b_0}{a_0}a_1.$$ Hence $b=\frac{b_0}{a_0}\bar{a}$.

Now we consider the case when $a_0=0$.
If $b_0 \ne 0$, by the above identities, $a_2=a_3=0$ and so $b_2=b_3=0$. Hence $a$ and $b$ are in $\C$.
If $b_0=0$, then $a$ and $b$ are purely imaginary. If $a_1 \ne 0$, then we get $b_3=\frac{b_1}{a_1}a_3$ and $b_2=\frac{b_1}{a_1}a_2$ from the above identities. Thus $b=\frac{b_1}{a_1}a=-\frac{b_1}{a_1}\bar{a}$. If $a_1=0$, then $a_3b_1=a_2b_1=0$ by the above identities. If $b_1\neq 0$, then $a_2=a_3=0$ and thus $a=0$. This contradicts to the assumption $a\neq 0$. 
Hence $b_1=0$. In this case, $a$ and $b$ are of the form $a=a_*j$ and $b=b_*j$ for some $a_*, b_* \in \C$. Therefore we complete the proof.
\end{proof}
%When we add more assumptions in above proposition, we get a little strong result.
%\begin{proposition}
%For two nonzero quaternions $a$ and $b$, if $ab$, $ba$, $a\bar{b}$, $\bar{a}b$ are all in $\C$, then $a$ and $b$ satisfy one of the following;
%\begin{itemize}
%\item[(i)] $a$,$b \in \C$
%\item[(ii)] $a$ and $b$ are of the form $a=a_2j+a_3k, b=b_2j+b_3k$, where $a_2, a_3, b_2, b_3 \in \R$
%\item[(iii)] $b=ra$ for some $r \in \R-\{0\}$ and $a$ and $b$ are purely imaginary
%\end{itemize}
%\end{proposition}
%\begin{proof}
%The proof is quite similar of that of above proposition. Since $a\bar{b}$ and  $\bar{a}b$ are also in $\C$, by considering j-parts and k-parts of $a\bar{b}$ and $\bar{a}b$, we also get $a_0b_2-a_2b_0=0$ and $a_0b_3-a_3b_0=0$.\\
%Hence in the case of $a_0 \ne 0$, $b_2=b_3=0$ and so $a_2=a_3=0$ since $b$ is nonzero. Therefore when $a_0 \ne 0$, $a,b \in \C$.\\
%In the case of $a_0=0$, it is exactly the same as the proof of the above proposition. 
%\end{proof}
The next lemma is quite elementary and the proof is easy by a straight computation.

\begin{lemma}\label{lem:aia}
For a quaternion $q$, if $qi\bar{q}$ and $\bar{q}iq$ are complex numbers, then either $q \in \C$ or $q$ is of the form $q=q_*j$ for some $q_* \in \C$.
\end{lemma}

\begin{proof}
Let $q=q_0+q_1i+q_2j+q_3k$ for $q_i \in \R$, $i=1,2,3,4$. Then by a straight computation, we have
$$
qi\bar{q}=(q_0^2+q_1^2-q_2^2-q_3^2)i+2(q_0q_3+q_1q_2)j-2(q_0q_2-q_1q_3)k \in \C,
$$
$$
\bar{q}iq=(q_0^2+q_1^2-q_2^2-q_3^2)i-2(q_0q_3-q_1q_2)j+2(q_0q_2+q_1q_3)k \in \C.
$$
Hence,  if $qi\bar{q}$ and $\bar{q}iq$ are complex numbers,  it follows that $$q_0q_3=q_1q_2=q_0q_2=q_1q_3=0.$$ If $q \not \in \C$, then $q_0=q_1=0$ and thus the lemma follows.
\end{proof}

\section{Proof of the main Theorem}
Let $G$ be a non-elementary discrete subgroup of $\mathrm{Sp}(2,1)$ in which the sum of the diagonal entries of each element of $G$ is a complex number.
Let $A$ be a loxodromic element of $G$ fixing $0$ and $\infty$, $B$ be an arbitrary element in $G$. In terms of matrices,
we write $A$ and $B$ as \begin{eqnarray}\label{matrixform} A=\left[\begin{matrix} \lambda\mu & 0 & 0 \\ 0 & \nu & 0 \\ 0 & 0 & \frac{\mu}{\lambda} \end{matrix} \right], \ B=\left[\begin{matrix} a & b & c \\ d & e & f \\ g & h & l \end{matrix} \right],\end{eqnarray} where $\mu,\nu \in \mathrm{Sp}(1)$ and $\lambda >1$. For more detail, see \cite{Kim} or \cite{KP}. 

%Here we choose $B$ such that $bd \ne 0$. If every element of $G$ satisfies $bd=0$, we will deal with this case later. We also note that $cg \ne 0$. If $cg=0$, then $B$ fixes $0$ or $\infty$, which contradicts the discreteness of $G$ because $A$ fixes $0$ and $\infty$. Then the starting point is the following lemma.

\begin{lemma}\label{lem:loxodromic}
The matrix $A$ of $G$ fixing $0$ and $\infty$ is an element of $\mathrm{U}(2,1)$. In other words, $\mu,\nu \in \mathrm{U}(1)$.
\end{lemma}
\begin{proof}
For a matrix $X$, we denote by $tr(X)$ the sum of the diagonal entries of $X$. 
Let $\mu=\mu_0+\mu_1i+\mu_2j+\mu_3k$ and $\nu=\nu_0+\nu_1i+\nu_2j+\nu_3k$ for $\mu_t,\nu_t \in \bR$ where $t=0,1,2,3$.
Then $$tr(A)=\left(\lambda+1/\lambda\right)(\mu_0+\mu_1i+\mu_2j+\mu_3k)+(\nu_0+\nu_1i+\nu_2j+\nu_3k) \in \C,$$
and hence \begin{align}\label{eqn:lem3.1.1} (\lambda+1/\lambda)\mu_t+\nu_t=0 \text{ for }t=2,3.\end{align} Furthermore, considering
\begin{align*}
tr(A^2) & =(\lambda^2+1/\lambda^2)\mu^2+\nu^2\\ &= (\lambda^2+1/\lambda^2)(\mu_0^2-\mu_1^2-\mu_2^2-\mu_3^2+2\mu_0\mu_1i+2\mu_0\mu_2j+2\mu_0\mu_3k)\\
&+(\nu_0^2-\nu_1^2-\nu_2^2-\nu_3^2+2\nu_0\nu_1i+2\nu_0\nu_2j+2\nu_0\nu_3k) \in \C,
\end{align*}
we have that for $t=2,3$, \begin{align}\label{eqn:lem3.1.2} (\lambda^2+1/\lambda^2)\mu_0\mu_t+\nu_0\nu_t=\mu_t[(\lambda^2+1/\lambda^2)\mu_0-(\lambda+1/\lambda)\nu_0]=0.\end{align} 
If $\mu_2=\mu_3=0$,  Equation (\ref{eqn:lem3.1.1}) implies that $\nu_2=\nu_3=0$. Then $\mu, \nu \in \C$ and so $\mu, \nu \in \mathrm{U}(1)$.\\
From now on, we assume that $\mu_2 \ne 0$ or $\mu_3 \ne 0$. By Equation (\ref{eqn:lem3.1.2}),   \begin{eqnarray}\label{mu0nu0}(\lambda^2+1/\lambda^2)\mu_0-(\lambda+1/\lambda)\nu_0=0,\end{eqnarray} and we can write $$\displaystyle{\nu=\frac{\lambda^4+1}{\lambda(\lambda^2+1)}\mu_0+\nu_1i-\left(\lambda+\frac{1}{\lambda}\right)(\mu_2j+\mu_3k)}.$$
Now let us consider $A^4$. Then
\begin{align*}
tr(A^4) & =\left(\lambda^4+\frac{1}{\lambda^4}\right)\mu^4+\nu^4\\ &= \left(\lambda^4+\frac{1}{\lambda^4}\right)(\mu_0+\mu_1i+\mu_2j+\mu_3k)^4\\ &+\left(\frac{\lambda^4+1}{\lambda(\lambda^2+1)}\mu_0+\nu_1i-\left(\lambda+\frac{1}{\lambda}\right) (\mu_2j+\mu_3k)\right)^4 \in \C.
\end{align*}
Since $\mu,\nu \in \mathrm{Sp}(1)$, we get
\begin{align*}
& \mu_0^2+\mu_1^2+\mu_2^2+\mu_3^2=1, \\
& {|\nu|^2= \left(\frac{\lambda^4+1}{\lambda(\lambda^2+1)}\mu_0\right)^2+\nu_1^2+\left(\lambda+\frac{1}{\lambda}\right)^2(\mu_2^2+\mu_3^2)=1}.
\end{align*}
Using these identities and calculating the $j$-part of $tr(A^4)$, we have that 
%\begin{align*}
%& 4\mu_0\mu_2\left(\lambda^4+\frac{1}{\lambda^4}\right)(\mu_0^2-\mu_1^2-\mu_2^2-\mu_3^2)\\
%& -\frac{4(\lambda^4+1)}{\lambda(\lambda^2+1)}\mu_0\left(\lambda+\frac{1}{\lambda}\right)\mu_2
%\left(\left(\frac{\lambda^4+1}{\lambda(\lambda^2+1)}\mu_0\right)^2-\nu_1^2-\left(\lambda+\frac{1}{\lambda}\right)^2(\mu_2^2+\mu_3^2)\right)\\
%&= 4\mu_0\mu_2\left(\lambda^4+\frac{1}{\lambda^4}\right)(2\mu_0^2-1)-\frac{4(\lambda^4+1)}{\lambda^2}\mu_0\mu_2\left(2\left(\frac{\lambda^4+1}{\lambda(\lambda^2+1)}\mu_0\right)^2-1\right)\\
%&= \frac{4\mu_0\mu_2(\lambda^2-1)(\lambda^6-1)(4\lambda^2\mu_0^2-(\lambda^2+1)^2)}{\lambda^4(\lambda^2+1)^2}=0.
%\end{align*}
$$\frac{4\mu_0\mu_2(\lambda^2-1)(\lambda^6-1)(4\lambda^2\mu_0^2-(\lambda^2+1)^2)}{\lambda^4(\lambda^2+1)^2}=0.$$
Since $\lambda >1$ and $0 \le \mu_0^2 < 1$, it follows that $\mu_0\mu_2=0$. By repeating the same argument for the $k$-part of $tr(A^4)$, one gets $\mu_0\mu_3=0$. Since $\mu_2 \ne 0$ or $\mu_3 \ne 0$, $\mu_0=0$ and hence $\nu_0=0$ by (\ref{mu0nu0}). That is, $\mu$ and $\nu$ are purely imaginary, and so $\bar{\mu}=-\mu$ and $\bar{\nu}=-\nu$. Since $|\mu|=|\nu|=1$, we know that $\mu^3=-\mu$ and $\nu^3=-\nu$. If we write $\displaystyle{\nu=\nu_1i-(\lambda+1/\lambda)(\mu_2j+\mu_3k)}$ as before, since 
$$
\displaystyle{tr(A^3)=\left(\lambda^3+\frac{1}{\lambda^3}\right)\mu^3+\nu^3=-\left(\lambda^3+\frac{1}{\lambda^3}\right)\mu-\nu \in \C},
$$
the $j$-part of $tr(A^3)$ is zero, i.e.,
$$
-\left(\lambda^3+\frac{1}{\lambda^3}\right)\mu_2+\left(\lambda+\frac{1}{\lambda}\right)\mu_2=-\left(\lambda+\frac{1}{\lambda}\right)\left(\lambda-\frac{1}{\lambda}\right)^2\mu_2=0.
$$
Since $\lambda>1$, we have $\mu_2=0$. Similarly, considering the $k$-part of $tr(A^3)$, we also get $\mu_3=0$. This contradicts to the assumption that  $\mu_2 \ne 0$ or $\mu_3 \ne 0$. Therefore $\mu_2=\mu_3=0$ and thus $\mu, \nu \in \mathrm{Sp}(1) \cap \C=\mathrm U(1)$.
\end{proof}

According to Lemma \ref{lem:loxodromic}, $A$ is written as
\begin{align*} A=\left[\begin{matrix} \lambda e^{i\theta} & 0 & 0 \\ 0 & e^{-2i\theta} & 0 \\ 0 & 0 & \frac{1}{\lambda}e^{i\theta} \end{matrix} \right], \ \lambda>1, \ \theta \in [0,2\pi).\end{align*}

\begin{lemma}\label{lemma:diagonal}
For any element $B \in G$, every diagonal entry of $B$  is a complex number.
%If $B=\left[\begin{matrix} a & b & c \\ d & e & f \\ g & h & l \end{matrix} \right]$ is an arbitrary element in $G$, then $a$, $e$, $l$ $\in \C$.
\end{lemma}
\begin{proof}
Let $B$ be the matrix $$B=\left[\begin{matrix} a & b & c \\ d & e & f \\ g & h & l \end{matrix} \right].$$
Since the sum of the diagonal entries of every element of $G$ is in $\C$, we have that 
\begin{align*}
& tr(B)=a+e+l \in \mathbb C, \\
& tr(AB)=\lambda e^{i\theta}a+e^{-2i\theta}e+\frac{e^{i\theta}}{\lambda}l \in \mathbb C, \\
& tr(A^{-1}B)=\frac{e^{-i\theta}}{\lambda}a+e^{2i\theta}e+\lambda e^{-i\theta}l \in  \mathbb C.
\end{align*}
Solving for $a,e,$ and $l$, we conclude that $a,e,l \in \C$. This shows that every element of $G$ has complex diagonal entries.
\end{proof}

\begin{lemma}\label{lem:3elements}
Let $A$, $B_1$ and $B_2$ be elements of $G$ and 
$$
A=\left[\begin{matrix} \lambda e^{i\theta} & 0 & 0 \\ 0 & e^{-2i\theta} & 0 \\ 0 & 0 & \frac{1}{\lambda}e^{i\theta}
\end{matrix}\right], B_1=\left[\begin{matrix} a_1 & b_1 & c_1 \\ d_1 & e_1 & f_1 \\ g_1 & h_1 & l_1
\end{matrix}\right], B_2=\left[\begin{matrix} a_2 & b_2 & c_2 \\ d_2 & e_2 & f_2 \\ g_2 & h_2 & l_2
\end{matrix}\right],
$$
for some $\lambda>1$, $\theta \in [0,2\pi)$. Then $b_1d_2, c_1g_2, d_1b_2, f_1h_2, g_1c_2, h_1f_2$ are all complex numbers. Furthermore
$b_1id_2, c_1ig_2, d_1ib_2, f_1ih_2, g_1ic_2, h_1if_2$ are complex numbers unless $\theta \equiv 0$ (mod $\frac{\pi}{2}$).
\end{lemma}
\begin{proof}
Since $B_1$ and $B_2$ are in $G$, we know that $a_1, e_1, l_1, a_2, e_2, l_2$ are all complex numbers by Lemma \ref{lemma:diagonal}. Consider the elements $B_1AB_2$, $B_1A^2B_2$, $B_1A^3B_2$, $B_1A^4B_2 \in G$ and then their $(1,1)$-entries:
$$
\lambda e^{i\theta}a_1a_2+b_1e^{-2i\theta}d_2+c_1e^{i\theta}g_2/\lambda,
$$
$$
\lambda^2 e^{2i\theta}a_1a_2+b_1e^{-4i\theta}d_2+c_1e^{2i\theta}g_2/\lambda^2,
$$
$$
\lambda^3 e^{3i\theta}a_1a_2+b_1e^{-6i\theta}d_2+c_1e^{3i\theta}g_2/\lambda^3,
$$
$$
\lambda^4 e^{4i\theta}a_1a_2+b_1e^{-8i\theta}d_2+c_1e^{4i\theta}g_2/\lambda^4.
$$
These are also all complex numbers by Lemma \ref{lemma:diagonal}. Since $a_1,a_2 \in \mathbb{C}$, the following are all complex numbers as well.
$$
\lambda \cos2\theta(b_1d_2)-\lambda \sin2\theta(b_1id_2)+\cos \theta(c_1g_2)+\sin \theta(c_1ig_2)=z_1,
$$
$$
\lambda^2 \cos4\theta(b_1d_2)-\lambda^2 \sin4\theta(b_1id_2)+\cos 2\theta(c_1g_2)+\sin 2\theta(c_1ig_2)=z_2,
$$
$$
\lambda^3 \cos6\theta(b_1d_2)-\lambda^3 \sin6\theta(b_1id_2)+\cos 3\theta(c_1g_2)+\sin 3\theta(c_1ig_2)=z_3,
$$
$$
\lambda^4 \cos8\theta(b_1d_2)-\lambda^4 \sin8\theta(b_1id_2)+\cos 4\theta(c_1g_2)+\sin 4\theta(c_1ig_2)=z_4.
$$
Solving for $b_1d_2,b_1id_2,c_1g_2,c_1ig_2$, we get that they are all complex numbers unless $\theta \equiv 0$ (mod $\frac{\pi}{2}$). It can be easily checked that $b_1d_2,c_1g_2$ are still complex numbers when $\theta \equiv 0$ (mod $\frac{\pi}{2}$). Hence we can conclude that $b_1d_2,c_1g_2$ are complex numbers.
Similarly, from the other diagonal entries of $B_1AB_2$, $B_1A^2B_2$, $B_1A^3B_2$, $B_1A^4B_2 \in G$, one can show the Lemma.
\end{proof}

Applying  Lemma \ref{lem:3elements} to $B_1=B_2=B$ and $B_1=B,B_2=B^{-1}$(or $B_2=B,B_1=B^{-1}$), we immediately have the following corollary.
\begin{corollary}\label{cor1}
Let $B$ be the matrices written in (\ref{matrixform}). Then,
\begin{itemize}
\item[(a)] $bd$, $db$, $fh$, $hf$, $cg$, $gc$, $b\bar{h}$, $f\bar{d}$, $c\bar{g}$, $\bar{h}b$, $\bar{d}f$, $\bar{g}c \in \C$.
\item[(b)] $bid$, $dib$, $fih$, $hif$, $cig$, $gic$, $bi\bar{h}$, $fi\bar{d}$, $ci\bar{g}$, $\bar{h}ib$, $\bar{d}if$, $\bar{g}ic \in \C$ unless $\theta \equiv 0$ (mod $\frac{\pi}{2}$).
\end{itemize}
\end{corollary}

%\begin{remark}
%As a summary of the previous results, for an arbitrary element $B$ in $G$ of the form (\ref{matrixform}), we know that the following $a, e, l, bd, db, fh, hf, cg, gc, b\bar{h}, f\bar{d}, c\bar{g}, \bar{h}b, \bar{d}f, \bar{g}c \in \C$. Moreover  $bid, dib, fih, hif, cig, gic, bi\bar{h}, fi\bar{d}, ci\bar{g}, \bar{h}ib, \bar{d}if, \bar{g}ic \in \C$ unless $\theta \equiv 0$ (mod $\frac{\pi}{2}$)
%\end{remark}

Let $B$ be an arbitrary element of $G$ which are not a power of $A$. 
Let $B$ be the matrix in (\ref{matrixform}). Suppose that $cg=0$.
Then $B$ fixes either $0$ or $\infty$ (see Remark \ref{remark}). This means that $A$ and $B$ have a common fixed point. However this is impossible since $G$ is discrete. Hence $cg\neq 0$.

\subsection{The $bd\neq 0$ case}

We will first deal with the case that there exists an element $B$ of $G$ with $bd \neq 0$. Throughout this section, we assume that $bd\neq 0$.
As seen in Corollary \ref{cor1}, both $bd$ and $db$ are complex numbers. Applying Proposition \ref{prop:commuting} for $b$ and $d$, one of the following holds.
\begin{itemize}
\item[(i)] $b$, $d \in \C$
\item[(ii)] $b$ and $d$ are of the form $b=b_{*}j$ and $d=d_{*}j$, where $b_{*}, d_{*} \in \C$
\item[(iii)] $d=r\bar{b}$ for some $r \in \R-\{0\}$
\end{itemize}
We will consider these cases separately as follow. %in each case of $bd\neq 0$ and $bd=0$.

\begin{case}
Suppose that $b, d \in \C$.
Since $b\bar{h}, \bar{h}b \in \C$ and $b$ is nonzero, $h \in \C$. Similarly, since $f\bar{d}, \bar{d}f \in \C$ and $d$ is nonzero, $f \in \C$. Then, from Equation (2) and (13), it can be seen that $c$ and $g$ are also complex numbers. 
Thus every entry of $B$ is a complex number and hence $B \in \mathrm{U}(2,1)$. 

Let $B'$ be any other element of $G$ which is not a power of $A$. Let $$B'=\left[\begin{matrix} a' & b' & c' \\ d' & e' & f' \\ g' & h' & l' \end{matrix} \right].$$ Then, by Lemma \ref{lemma:diagonal}, we know that $a',e',l' \in \C$. 
Furthermore, by Lemma \ref{lem:3elements}, we have that $$bd',db',b\bar{h'},f'\bar{d},cg',gc' \in \C.$$
Since $b$ and $d$ are nonzero complex numbers, it follows that $b',d',h',f' \in \C$. 
Moreover from Equation (2) and (13), it follows that $c', g' \in \C$. Thus every entry of $B'$ is a complex number, i.e. $B'\in \mathrm U(2,1)$. Therefore $G$ is a subgroup of $\mathrm{U}(2,1)$, which preserves a copy of $\ch{2}$ in $\mathbf H_{\mathbb H}^2$.
\end{case}

\begin{case} Now we suppose that $b=b_{*}j$ and $d=d_{*}j$ for some $b_*, d_* \in \C$. 
Since $b\bar{h}, \bar{h}b \in \C$ by Corollary \ref{cor1} and $b$ is nonzero, $h=h_{*}j$ for some $h_{*} \in \C$. In the same way, since $f\bar{d}, \bar{d}f \in \C$ and $d$ is nonzero, $f=f_{*}j$ for some $f_{*} \in \C$. 
Furthermore, by Equation (2), we have that $$a\bar f +b \bar e + c \bar d = -af_*j+b_*j \bar e -cd_*j.$$
Since $e$ is a complex number, $j\bar{e}=ej$. Hence 
$$-af_*j+b_*j \bar e-cd_*j=(-af_*+b_*e-cd_*)j=0.$$
Due to $d_* \neq 0$, we conclude that $c$ is a complex number.
Similarly, from Equation (13), it can be derived that $g \in \C$. To summarize, $a,e,l,c,g \in \C$ and $b,d,f,h$ are of the form  $q_{*}j$ where $q_{*} \in \C$. Then, for $z_1, z_2 \in \C$,
$$
A\left[\begin{matrix} z_1 \\ z_2j \\ 1 \end{matrix} \right]=\left[\begin{matrix} \lambda e^{i\theta}z_1 \\ e^{-2i\theta}z_2j \\ \frac{1}{\lambda}e^{i\theta} \end{matrix} \right] \sim \left[\begin{matrix} z_1' \\ z_2'j \\ 1 \end{matrix} \right]
$$
for some $z_1', z_2' \in \C$ and
\begin{align*}
B\left[\begin{matrix} z_1 \\ z_2j \\ 1 \end{matrix} \right] &=\left[\begin{matrix} a & b_{*}j & c \\ d_{*}j & e & f_{*}j \\ g & h_{*}j & l \end{matrix} \right]\left[\begin{matrix} z_1 \\ z_2j \\ 1 \end{matrix} \right]=\left[\begin{matrix} az_1+b_{*}jz_2j+c \\ d_{*}jz_1+ez_2j+f_{*}j \\ gz_1+h_{*}jz_2j+l \end{matrix} \right]\\ &=\left[\begin{matrix} az_1-b_{*}\bar{z_2}+c \\ (d_{*}\bar{z_1}+ez_2+f_{*})j \\ gz_1-h_{*}\bar{z_2}+l \end{matrix} \right] \sim \left[\begin{matrix} z_1'' \\ z_2''j \\ 1 \end{matrix} \right]
\end{align*}
for some $z_1'', z_2'' \in \C$. Note that $\left[\begin{matrix} z_1 & z_2j & 1 \end{matrix} \right]=\left[\begin{matrix} z_1 & j\bar{z_2} & 1 \end{matrix} \right]$ for $z_1,z_2 \in \C$.
Hence $A$ and $B$ leave invariant a copy of $\ch{2}$ of polar vectors $\left[\begin{matrix} z_1 & jz_2 & 1 \end{matrix} \right]^t$, where $z_1,z_2 \in \C$.
 
Let $B'$ be any other element of $G$ which are not a power of $A$ and let \begin{align*} B'=\left[\begin{matrix} a' & b' & c' \\ d' & e' & f' \\ g' & h' & l' \end{matrix} \right].\end{align*} Then, $a',e',l' \in \C$ by Lemma \ref{lemma:diagonal}. Applying Lemma \ref{lem:3elements} to $B$ and $B'$, one can conclude that $bd'$, $d'b$, $b'd$, $db' \in \C$. Then, by Proposition \ref{prop:commuting}, $b'$ and $d'$ are of the form $q_{*}j$ where $q_{*} \in \C$ since $b, d$ are of the form $q_*j$ for $q_*\in \C$. By a similar argument, one can show that $f'$ and $h'$ are of the same form $q_*j$. Moreover $c',g' \in \C$ because $cg',gc' \in \C$ by Lemma \ref{lem:3elements} and $cg\neq 0$. Therefore $B'$ is of the same form as $B$ and we conclude that every element of $G$ preserves a copy of $\ch{2}$ consisting of $\left[\begin{matrix} z_1 & jz_2 & 1 \end{matrix} \right]^t$, where $z_1,z_2 \in \C$.
\end{case}

\begin{case}
Lastly suppose that $d=r\bar{b}$.
To avoid repetition, we will assume that $b$ and $d$ are neither complex numbers nor of the form $b=b_*j$ and $d=d_*j$ for $b_*,d_* \in \C$.
Let $d=r_1\bar{b}$ for some $r_1 \in \mathbb{R}-\{0\}$. By Corollary \ref{cor1}, we have that $f\bar{d}$, $\bar{d}f$, $ b\bar{h}$, $\bar{h}b \in \C$. Applying Proposition \ref{prop:commuting} to $f,\bar{d}$ and $b,\bar{h}$ respectively, it can be easily seen that $f=r_2'd=r_2'r_1\bar b=r_2 \bar b$ and $h=r_3b$ for some $r_2,r_3 \in \mathbb{R}-\{0\}$. From (5) and (14), 
$$2r_1r_2|b|^2=2r_3|b|^2=1-|e|^2,$$ and thus, $$\displaystyle{r_3=r_1r_2=\frac{1-|e|^2}{2|b|^2}} \text{ and } h=r_1r_2b.$$ Moreover using (2), (4), (11), (13), we have the following equations:
\begin{align*}
& r_2ab+r_1cb+b\bar{e}=0, & & r_1lb+r_2gb+r_1r_2b\bar{e}=0,  \\
&r_1r_2\bar{c}b+\bar{l}b+r_2be=0,& &  r_1r_2\bar{a}b+\bar{g}b+r_1be=0.
\end{align*}
These equations are written as 
\begin{align*}
& -r_1r_2b\bar{e}=r_1r_2^2ab+r_1^2r_2cb=r_1lb+r_2gb, \\
& -r_1r_2be=r_1^2r_2\bar{c}b+r_1\bar{l}b=r_1r_2^2\bar{a}b+r_2\bar{g}b.
\end{align*}
Since $b \ne 0$, the above equations are simplified in the following way.
\begin{align*}
& r_1r_2^2a+r_1^2r_2c=r_1l+r_2g, \\
& r_1r_2^2a-r_1^2r_2c=r_1l-r_2g.
\end{align*}
Hence we finally get that $r_1r_2^2a=r_1l$ and $r_1^2r_2c=r_2g$, i.e. $l=r_2^2a, g=r_1^2c$ since $r_1,r_2 \ne 0$. Now $B$ is written as 
\begin{align} B=\left[\begin{matrix} a & b & c \\ r_1\bar{b} & e & r_2\bar{b} \\ r_1^2c & r_1r_2b & r_2^2a \end{matrix} \right], \text{ where } a,e \in \C \text{ and } r_1,r_2 \in \R-\{0\}.\end{align}  %We also note that $r_2^2=1$ because $B^{-1}=\left[\begin{matrix} r_2^2\bar{a} & r_2b & \bar{c} \\ r_1r_2\bar{b} & \bar{e} & \bar{b} \\ r_1^2\bar{c} & r_1b & \bar{a} \end{matrix} \right]$ and $r_1r_2\times\{$(1,2)-entry$\}=\{$(3,2)-entry$\}$. Therefore $B=\left[\begin{matrix} a & b & c \\ r_1\bar{b} & e & r_2\bar{b} \\ r_1^2c & r_1r_2b & a \end{matrix} \right]$, where $a,e \in \C$, $r_1 \in \R-\{0\}$ and $r_2=\pm1$.\\
Since $cg=r_1^2c^2 \in \C$ by Corollary 3.4, either $c \in \C$ or $c$ is purely imaginary.\\
Before considering these two cases, note that $\theta \equiv 0$ (mod $\frac{\pi}{2}$). If $\theta \not\equiv 0$ (mod $\frac{\pi}{2}$), then $bid=r_1bi\bar{b} \in \C$ and $dib=r_1\bar{b}ib \in \C$ by Corollary \ref{cor1}. Hence, by Lemma \ref{lem:aia}, either $b \in \C$ or $b=b_*j$ for $b_* \in \C$, which contradicts to our assumption.

\vspace{3mm}

\noindent {\bf Case 3.1 : $c\in \C$}

\vspace{3mm}

From (11), $(r_2\bar{a}+r_1\bar{c})b+be=0$ and thus we have that $$\displaystyle{e=-\frac{\bar{b}(r_2\bar{a}+r_1\bar{c})b}{|b|^2} \in \C}.$$ Writing $a=a_0+a_1i, c=c_0+c_1i$ and $b=b_0+b_1i+b_2j+b_3k$, the $j$-part and the $k$-part of $-\bar{b}(r_2\bar{a}+r_1\bar{c})b$ are       
$2(r_2a_1+r_1c_1)(b_1b_2-b_0b_3)$ and $2(r_2a_1+r_1a_1)(b_0b_2+b_1b_3)$ respectively, and they should be zero for $e \in \C$.\\

\smallskip

\noindent {\bf Claim} : $e \in \R$.
\begin{proof}[Proof of the claim]
It is sufficient to prove that $r_2\bar{a}+r_1\bar{c}$ is a real number.
Suppose not, i.e. $\mathrm{Im}(r_2\bar{a}+r_1\bar{c})=-(r_2a_1+r_1c_1) \ne 0$. Then by the above argument, \begin{align}\label{2eqn} b_1b_2=b_0b_3 \text{ and } b_0b_2=-b_1b_3.\end{align} 
If $b_2=0$, then $b_3 \ne 0$ since $b \not\in \C$. Then it follows from (\ref{2eqn}) that $b_0=b_1=0$. However, this contradicts to the assumption that $b$ is not of the form $b_*j$, where $b_* \in \C$. Hence $b_2 \ne 0$. In a similar way, it can be easily shown that $b_3 \ne 0$.
If $b_0=0$, then $b_1 \ne 0$ since $b$ is not of the form $b_*j$ where $b_* \in \C$ and then $b_2=b_3=0$ by (\ref{2eqn}), i.e. $b \in \C$. Again this contradicts to the assumption that $b$ is a complex number. Hence $b_0 \ne 0$. In a similar way, it can be seen  that $b_1 \ne 0$.
Therefore $b_0,b_1,b_2,b_3$ are all nonzero. From (\ref{2eqn}), it holds that $b_0b_1b_2^2=-b_0b_1b_3^2$ and thus $b_2^2=-b_3^2$. However it is impossible for any nonzero real numbers $b_2, b_3$.
Therefore we conclude that $\mathrm{Im}(r_2\bar{a}+r_1\bar{c})=0$ and thus $r_2\bar{a}+r_1\bar{c}$ is a real number, which completes the proof.
\end{proof}

Since $\theta \equiv 0$ (mod $\frac{\pi}{2}$) as mentioned before, $e^{-2i\pi} \in \R$.
%Now consider  $A=\left[\begin{matrix} \lambda e^{i\theta} & 0 & 0 \\ 0 & e^{-2i\theta} & 0 \\ 0 & 0 & \frac{1}{\lambda}e^{i\theta} \end{matrix} \right]$ again. Since $\theta \equiv 0$ (mod $\frac{\pi}{2}$), $e^{-2i\pi} \in \R$.
Therefore, for any $z_1,z_2,z_3 \in \C$, $$A\left[\begin{matrix} z_1 \\ \bar{b}z_2 \\ z_3 \end{matrix}\right]=\left[\begin{matrix} \lambda e^{i\theta} & 0 & 0 \\ 0 & e^{-2i\theta} & 0 \\ 0 & 0 & \frac{1}{\lambda}e^{i\theta} \end{matrix} \right] \left[\begin{matrix} z_1 \\ \bar{b}z_2 \\ z_3 \end{matrix}\right]= \left[\begin{matrix} z'_1 \\ \bar{b}z'_2 \\ z'_3 \end{matrix}\right]$$ for some $z'_1,z'_2,z'_3 \in \C$. In addition,
\begin{align*}
B\left[\begin{matrix} z_1 \\ \bar{b}z_2 \\ z_3 \end{matrix} \right] &=\left[\begin{matrix} a & b & c \\ r_1\bar{b} & e & r_2\bar{b} \\ r_1^2c & r_1r_2b & r_2^2a \end{matrix} \right]\left[\begin{matrix} z_1 \\ \bar{b}z_2 \\ z_3 \end{matrix} \right]\\ &=\left[\begin{matrix} az_1+|b|^2z_2+cz_3 \\ \bar{b}(r_1z_1+ez_2+r_2z_3) \\ r_1^2cz_1+r_1r_2|b|^2z_2+r_2^2az_3 \end{matrix} \right] = \left[\begin{matrix} z_1'' \\ \bar{b}z_2'' \\ z_3'' \end{matrix} \right]
\end{align*}
for some $z''_1,z''_2,z''_3 \in \C$.

Let $B'$ be any other element of $G$ which are not a power of $A$. By applying Lemma \ref{lem:3elements} to $B$ and $B'$ as in the previous case, one can check that $B'$ has the same form as $B$, i.e. 
\begin{align*} B'=\left[\begin{matrix} a' & b' & c' \\ r_3\bar{b'} & e' & r_4\bar{b'} \\ r_3^2c' & r_3r_4b' & r_4^2a' \end{matrix} \right], \text{ where } a', c' \in \C, e' \in \R \text{ and } r_3,r_4 \in \R-\{0\}.\end{align*} 
%$B'=\left[\begin{matrix} a' & b' & c' \\ r_3\bar{b'} & e' & r_4\bar{b'} \\ r_3^2c' & r_3r_4b' & a' \end{matrix} \right]$, where $a',c'\in \C$, $e' \in \R$, $r_3 \in \R-\{0\}$ and $r_4=\pm1$. 
Then, considering the diagonal entries of $B'B$, it follows that 
\begin{align*} & a'a+r_1b'\bar{b}+r_1^2c'c \in \C, \\ & r_1\bar{b}b'+ee'+r_2r_3r_4\bar{b}b' \in \C. \end{align*}
Since $a,a',c,c' \in \C$, $e, e' \in \R$ and $r_1, r_2, r_3, r_4 \ne 0$, we have that $b'\bar{b} \in \C$ and $\bar{b}b' \in \C$. Applying Proposition \ref{prop:commuting} for $\bar{b}$ and $b'$, we have $b'=rb$ for some $r \in \R$
since $b$ is neither a complex number nor of the form $b=b_*j$ for some $b_* \in \C$.
Hence $$B'=\left[\begin{matrix} a' & rb & c' \\ r_3r\bar{b} & e' & r_4r\bar{b} \\ r_3^2c' & r_3r_4rb & r_4^2a' \end{matrix} \right].$$ Then it is easy to see that $B'$ also preserves a copy of $\ch{2}$ of polar vectors $\left[\begin{matrix} z_1 & \bar{b}z_2 & z_3 \end{matrix}\right]^t$. Therefore we conclude that $G$ preserves a copy of $\ch{2}$ of polar vectors $\left[\begin{matrix} z_1 & \bar{b}z_2 & z_3 \end{matrix}\right]^t$, where $z_1,z_2,z_3 \in \C$.

\vspace{3mm}

\noindent {\bf Case 3.2 : $c$ is purely imaginary}

\vspace{3mm}

Now we suppose that the previous case does not happen for any element of $G$.
Hence assume that $c$ is not a complex number.

\vspace{3mm}

\noindent {\bf Claim}. $r_2=-1$.

\begin{proof}[Proof of Claim]
Putting $a=a_0+a_1i$ and $c=c_1i+c_2j+c_3k$, the identity $a\bar{c}+|b|^2+c\bar{a} =0$ of (3) implies that 
$$|b|^2+2a_1c_1=0.$$
By a straight computation, the $(1,2)$-entry of $B^2$ is 
$$ab+be+r_1r_2cb.$$ 
From Equation (11), we have that $be=-r_2\bar a b-r_1\bar c b=-r_2\bar a b+r_1 c b$. The last equation follows from the assumption that $c$ is purely imaginary. Then the $(1,2)$-entry of $B^2$ is written as $$ab+be+r_1r_2cb=(a-r_2 \bar a+r_1(r_2+1)c)b.$$
Note that $a$ is a complex number, $c$ is purely imaginary and $b\neq 0$. Hence if $r_2 \neq -1$, the $(1,2)$-entry of $B^2$ can never be zero. In a similar way, one can see that the $(2,1)$-entry of $B^2$ is also nonzero.
Hence the $(1,3)$-entry of $B^2$ must be purely imaginary but not a complex number since we assume that the previous case does not happen. The $(1,3)$-entry of $B^2$ is $ac+r_2|b|^2+r_2^2ca$ and hence its real part is computed as 
$$2\mathrm{Re}(ac+r_2|b|^2+r_2^2ca)=2(-a_1c_1+r_2|b|^2-r_2^2a_1c_1)=|b|^2(r_2+1)^2=0.$$
Since $|b|\neq 0$, it follows that $r_2=-1$. This contradicts to the assumption that $r_2 \neq -1$. Therefore we conclude that $r_2=-1$.
\end{proof}

Now $B$ is written as \begin{align}\label{Bmatrixform} B=\left[\begin{matrix} a & b & c \\ r_1\bar{b} & e & -\bar{b} \\ r_1^2c & -r_1b & a \end{matrix} \right], \text{ where } a,e \in \C \text{ and } r_1\in \R-\{0\}.\end{align}  
We look at the matrix $BA$.
\begin{align*}
BA & =\left[\begin{matrix} a & b & c \\ r_1\bar{b} & e & -\bar{b} \\ r_1^2c & -r_1b & a \end{matrix} \right]\left[\begin{matrix} \lambda e^{i\theta} & 0 & 0 \\ 0 & e^{-2i\theta} & 0 \\ 0 & 0 & \frac{1}{\lambda}e^{i\theta} \end{matrix} \right] \\
&=\left[\begin{matrix} \lambda ae^{i\theta} & be^{-2i\theta} & ce^{i\theta}/\lambda \\ \lambda r_1\bar{b} e^{i\theta} & ee^{-2i\theta} & -\bar{b}e^{i\theta}/\lambda \\ \lambda r_1^2ce^{i\theta} & -r_1be^{-2i\theta} & ae^{i\theta}/\lambda \end{matrix} \right].
\end{align*}

Since $\theta \equiv 0$ (mod $\frac{\pi}{2}$) and so $e^{-2i\theta} \in \mathbb R$, the $(1,2)$-entry of $BA$ is neither a complex number nor of the form $q_*j$ for $q_*\in \C$. Hence $BA$ is of the same form as $B$ in (\ref{Bmatrixform}). Then the modulus of the $(1,2)$-entry of $BA$ should equal to the modulus of the $(2,3)$-entry of $BA$. Hence, we have that
$$|be^{-2i\theta}|=\left|\frac{-\bar b e^{i\theta}}{\lambda}\right|, \ |b|=\frac{|b|}{\lambda} \text{ and so }\lambda=1.$$
However, this contradicts to the assumption that $\lambda >1$. Therefore the case that $c$ is purely imaginary and not a complex number can never happen.
\end{case}

\subsection{The $bd=0$ case}
We look at the case that there exists an element $B$ of $G$ with $bd \ne 0$ so far. From now on, we consider the remaining case that every element of $G$ satisfies $bd=0$.
If $bd=0$, by considering $B^{-1}$, we also have $fh=0$. Then, using the identities (1)--(18), it can be easily checked that $b=d=f=h=0$. For example, if $b=f=0$, by (6), $d=0$ because $c \ne 0$. Then, by (15), $h=0$. Therefore every element of $G$ is of the form
$$\left[\begin{matrix} a & 0 & c \\ 0 & e & 0 \\ g & 0 & l \end{matrix} \right], \text{ where }a,e,l \in \C.$$
Applying Proposition \ref{prop:commuting} for $c$ and $g$, since $c, g \ne 0$, one of the following holds.
\begin{itemize}
\item[(i)] $c$, $g \in \C$
\item[(ii)] $c$ and $g$ are of the form $c=c_{*}j$ and $g=g_{*}j$ where $c_{*}, g_{*} \in \C$
\item[(iii)] $g=r\bar{c}$ for some $r \in \R-\{0\}$
\end{itemize}

First, if $c$, $g \in \C$, then $B \in \mathrm{U}(2,1)$. For any other element \begin{align}\label{B'} B'=\left[\begin{matrix} a' & 0 & c' \\ 0 & e' & 0 \\ g' & 0 & l' \end{matrix} \right] \in G, \text{ where }a',e',l' \in \C, \end{align} $c',g' \in \C$ since $cg',gc' \in \C$ by Lemma \ref{lem:3elements}. Hence $B' \in \mathrm{U}(2,1)$. This implies that $G$ is a subgroup of $\mathrm{U}(2,1)$.

Second, if $c$ and $g$ are of the form $c=c_{*}j$ and $g=g_{*}j$ where $c_{*}, g_{*} \in \C$, for $z_1,z_2,z_3 \in \C$,
\begin{align*}
& A\left[\begin{matrix} z_1j \\ z_2 \\ z_3 \end{matrix} \right] =\left[\begin{matrix} \lambda e^{i\theta} & 0 & 0 \\ 0 & e^{-2i\theta} & 0 \\ 0 & 0 & \frac{1}{\lambda}e^{i\theta}
\end{matrix}\right]\left[\begin{matrix} z_1j \\ z_2 \\ z_3 \end{matrix} \right] =\left[\begin{matrix} \lambda e^{i\theta}z_1j \\ e^{-2i\theta}z_2 \\ \frac{1}{\lambda}e^{i\theta}z_3 \end{matrix} \right]=\left[\begin{matrix} z_1'j \\ z_2' \\ z_3' \end{matrix} \right],\\
& B\left[\begin{matrix} z_1j \\ z_2 \\ z_3 \end{matrix} \right]=\left[\begin{matrix} a & 0 & c \\ 0 & e & 0 \\ g & 0 & l \end{matrix} \right]\left[\begin{matrix} z_1j \\ z_2 \\ z_3 \end{matrix} \right]=\left[\begin{matrix} az_1j+cz_3 \\ ez_2 \\ gz_1j+lz_3 \end{matrix} \right] =\left[\begin{matrix} z_1''j \\ z_2'' \\ z_3'' \end{matrix} \right],
\end{align*}
for some $z_1',z_2',z_3',z_1'',z_2'',z_3'' \in \C$. Hence, $A$ and $B$ leave invariant a copy of $\ch{2}$ of polar vectors $\left[\begin{matrix} z_1j & z_2 & z_3 \end{matrix} \right]^t$, where $z_1,z_2,z_3 \in \C$.

For any other element $B'\in G$ of the form (\ref{B'}), since $cg',gc' \in \C$ by Lemma \ref{lem:3elements}, $c',g'$ are also of the form $c'=c'_*j$, $g'=g'_*j$ for $c'_*$, $g'_* \in \C$. Therefore, every element of $G$ preserves a copy of $\ch{2}$ of polar vectors $\left[\begin{matrix} z_1j & z_2 & z_3 \end{matrix} \right]^t$, where $z_1,z_2,z_3 \in \C$.

Lastly, in the case that $g=r\bar{c}$ for some $r \in \R-\{0\}$, we assume that neither $c \in \C$ nor $c$ is of the form $c=c_*j$ for $c_* \in \C$ to avoid repetition. From (1), we have that $a\bar{l}+rc^2=1$ and so $c^2 \in \C$. Then $c$ should be purely imaginary because $c \not\in \C$. By (3), we have $\mathrm{Re}(ca)=0$, so $a \in \R$. Then, for $z_1,z_2,z_3 \in \C$,
\begin{align*}
B\left[\begin{matrix} cz_1 \\ z_2 \\ z_3 \end{matrix} \right] =\left[\begin{matrix} a & 0 & c \\ 0 & e & 0 \\ r\bar{c} & 0 & l \end{matrix} \right]\left[\begin{matrix} cz_1 \\ z_2 \\ z_3 \end{matrix} \right]=\left[\begin{matrix} c(az_1+z_3) \\ ez_2 \\ r|c|^2z_1+lz_3 \end{matrix} \right] =\left[\begin{matrix} cz_1' \\ z_2' \\ z_3' \end{matrix} \right],
\end{align*}
for some $z_1',z_2',z_3' \in \C$.

\vspace{3mm}
 
\noindent {\bf Claim}. $\theta \equiv 0$ (mod $\pi$).
\begin{proof}[Proof of Claim]
The $(1,1)$-entry of $BAB$ is a complex number, i.e.
$$\lambda e^{i\theta}a^2+\frac{rce^{i\theta}\bar{c}}{\lambda} \in \C.$$
Since $\lambda e^{i\theta}a^2 \in \C$, $ce^{i\theta}\bar{c}=|c|^2\cos\theta-(cic)\sin\theta \in \C$. Then, if $\theta \not\equiv 0$ (mod $\pi$), $cic \in \C$. Then, by Lemma \ref{lem:aia}, either $c \in \C$ or $c=c_*j$ for $c_* \in \C$. This contradicts to our assumption. Thus $\theta \equiv 0$ (mod $\pi$).
\end{proof}
Due to the claim above, $A$ is written as $$A=\left[\begin{matrix} \pm \lambda & 0 & 0 \\ 0 & 1 & 0 \\ 0 & 0 & \pm\frac{1}{\lambda}
\end{matrix}\right].$$Then, for $z_1,z_2,z_3 \in \C$,
\begin{align*}
A\left[\begin{matrix} cz_1 \\ z_2 \\ z_3 \end{matrix} \right] =\left[\begin{matrix} \pm c\lambda z_1 \\ z_2 \\ \pm\frac{1}{\lambda}z_3 \end{matrix} \right]=\left[\begin{matrix} cz_1' \\ z_2' \\ z_3' \end{matrix} \right],
\end{align*}
for some $z_1',z_2',z_3' \in \C$.

Thus $A$ and $B$ leave invariant a copy of $\ch{2}$ of polar vectors $\left[\begin{matrix} cz_1 & z_2 & z_3 \end{matrix} \right]^t$, where $z_1,z_2,z_3 \in \C$.
For any other element $B'\in G$ of the form (\ref{B'}), by Lemma \ref{lem:3elements}, $cg' \in \C$ and $g'c \in \C$. Since $c$ is purely imaginary, Proposition \ref{prop:commuting} implies that $g'=r'c$ for some $r' \in \R-\{0\}$ and $g'$ is purely imaginary. Since $c'g', g'c' \in \C$ by Corollary \ref{cor1}, $c'=r''g'$ for some $r''\in \R-\{0\}$ by Proposition \ref{prop:commuting}. Also, by a similar argument as above, we also have $a' \in \R$ using (3). Therefore, $B'$ is written as $$B'=\left[\begin{matrix} a' & 0 & r'r''c \\ 0 & e' & 0 \\ r'c & 0 & l' \end{matrix} \right],$$ where $e',l' \in \C$, $a' \in \R$, $r',r'' \in \R-\{0\}$, and $c$ is purely imaginary. Then, for $z_1,z_2,z_3 \in \C$,
\begin{align*}
B'\left[\begin{matrix} cz_1 \\ z_2 \\ z_3 \end{matrix} \right] =\left[\begin{matrix} a' & 0 & r'r''c \\ 0 & e' & 0 \\ r'c & 0 & l' \end{matrix} \right]\left[\begin{matrix} cz_1 \\ z_2 \\ z_3 \end{matrix} \right]=\left[\begin{matrix} c(a'z_1+r'r''z_3) \\ e'z_2 \\ -r'|c|^2z_1+l'z_3 \end{matrix} \right] =\left[\begin{matrix} cz_1' \\ z_2' \\ z_3' \end{matrix} \right],
\end{align*}
for some $z_1',z_2',z_3' \in \C$.
Therefore, every element of $G$ preserves a copy of $\ch{2}$ of polar vectors $\left[\begin{matrix} cz_1 & z_2 & z_3 \end{matrix} \right]^t$, where $z_1,z_2,z_3 \in \C$.

\vspace{3mm}

{\bf Acknowledgement.} 
The second author thanks to John Parker for useful discussions during staying in Hunan University and Yeuping Jiang for the hospitality staying in Changsha.


\begin{thebibliography}{99}
%    \bibitem{AK} B. N. Apanasov, I. Kim, Cartan angular invariant and deformations of rank 1 symmetric spaces,
 %   Sbornik: Mathematics 198:2 (2007), 147-169.
 %   \bibitem{Cao} W. Cao, Congruence classes of points in quaternionic hyperbolic space, Geom. Dedicata 180 (2016), 203–228.
 %   \bibitem{CPW} W. Cao, J. R. Parker, X. Wang, On the classification of quaternionic M$\ddot{o}$bius transformations,
  %  Math. Proc. Camb. Phil. Soc. 137 (2004), 349-359.
   % \bibitem{FLW} X. Fu, L. Li, X. Wang, A characterization of Fuchsian groups acting on complex hyperbolic spaces. Czechoslovak Math. J. 62(137) (2012), no. 2, 517-525.
    %\bibitem{Go} W. M. Goldman, Complex hyperbolic Geometry, Oxford Univ. Press, (1999).
    \bibitem{CG} H. Cunha and N. Gusevskii, \emph{A note on trace fields of complex hyperbolic groups}, Groups Geom. Dyn. 8 (2014), 355--374.
    \bibitem{FLW12} X. Fu, L. Li and X. Wang, \emph{A characterization of Fuchsian groups acting on complex hyperbolic spaces}, Czechoslovak Math. J. 62 (137) (2012), no. 2, 517--525.
    %\bibitem{FX} X. Fu and B. Xie, \emph{A characterization of Fuchsian group in $\mathrm{SU}(n,1)$}, Complex Var. Elliptic Equ. 59 (2014), no. 5, 723--731.
    \bibitem{Ge} J. Genzmer, \emph{Trace fields of subgroups of $\mathrm{SU}(n,1)$}, Acta Math. Vietnam 39 (2014), no. 3, 313--323.
    \bibitem{Kim} D. Kim, \emph{Discreteness criterions of isometric subgroups for quaternionic hyperbolic space}, Geom. Dedicata 106 (2004), 51--78.
    \bibitem{KP} I. Kim and J. R. Parker, \emph{Geometry of quaternionic hyperbolic manifolds},
    Math. Proc. Camb. Phil. Soc. 135 (2003), 291--320.
    \bibitem{JKim} J. Kim, \emph{Quaternionic hyperbolic Fuchsian groups},  Linear Algebra Appl. 438 (2013), 3610--3617. 
    \bibitem{KK14} J. Kim and S. Kim, \emph{A characterization of complex hyperbolic Kleinian groups in dimension $3$ with trace fields contained in $\mathbb R$}, Linear Algebra Appl. 455 (2014), 107--114.
    \bibitem{KK} S. Kim and J. Kim, \emph{Complex and quaternionic hyperbolic Kleinian groups with real trace fields}, J. Lond. Math. Soc. (2) 93 (2016), no. 1, 101--122.
    \bibitem{Mc} D. B. McReynolds, \emph{Arithmetic lattices in $\mathrm{SU}(n,1)$}, Book in preparation, available at http://www.its.caltech.edu/~dmcreyn/ComplexArithmeticI.pdf.
    \bibitem{NR} W. D. Neumann and A. Reid, \emph{Arithmetic of hyperbolic manifolds}, in Topology `90 (Columbus, OH, 1990), Ohio State University MSRI Publ. 1, de Gruyter, Berlin (1992), 273--310.
    %\bibitem{Ma} B. Maskit, Kleinian groups, Springer-Verlag, (1988).

\end{thebibliography}
\end{document}